\documentclass[12pt, a4paper]{article}
\usepackage{amsthm}
\usepackage{amsmath}
\usepackage{amssymb}
    \usepackage{bm}
\usepackage{fullpage}
\usepackage{enumerate}
\usepackage{mathrsfs}
\usepackage[all]{xy}
\usepackage[usenames,dvipsnames]{color}
\usepackage{graphicx}
\usepackage{bbm}
\usepackage[utf8]{inputenc}
\usepackage{enumitem}
\usepackage{multicol}
\usepackage[affil-it]{authblk}

\newtheorem{theorem}{Theorem}[section]
\newtheorem{remark}{Remark}[section]
\newtheorem{corollary}{Corollary}[section]
\newtheorem{lemma}{Lemma}[section]

\newtheorem{proposition}{Proposition}[section]
\newtheorem{example}{Example}[section]

\newtheorem{prescription}{Prescription}[section]

\newtheorem{definition}{Definition}[section]

\newcommand{\R}{\mathbb{R}}

\newcommand{\C}{\mathbb{C}}

\newcommand{\Q}{\mathbb{Q}}

\newcommand{\N}{\mathbbm{N}}

\newcommand{\E}{\mathbbm{E}}

\newcommand{\Levy}{L\'{e}vy\text{ }}

\newcommand{\dP}{\mathbbm{P}}
\newcommand{\mcl}{\mathcal}

\providecommand{\keywords}[1]{\textbf{\textit{Keywords:}} #1}
\newcommand{\mypar}{\par\text{ }}
\newcommand{\Cov}{\mathrm{Cov}}

\newcommand{\cccdot}{\hspace{-.05cm}  \cdot \hspace{-.075cm}}

\allowdisplaybreaks

\title{Association and other forms of positive dependence for {Feller} evolution systems}
\author{Eddie Tu%
 \thanks{\underline{Institution}: Dickinson College; \underline{Postal Address}:  Dickinson College, Department of Mathematics and Computer Science, PO Box 1773, Carlisle, PA 17013; \underline{Electronic address}: tue@dickinson.edu.
 }}

\begin{document}

\maketitle

\begin{abstract}
We prove characterizations of positive dependence 
for a general class of time-inhomogeneous Markov processes called Feller evolution processes (FEPs) and for jump-FEPs. General FEPs can be analyzed through their time and state-space dependent (extended) generators.  We will use the time and state-space dependent (extended) generators and time and state-space dependent \Levy measures to characterize the positive dependence of general FEPs and jump-FEPs, respectively. We conclude with applications of these results to additive processes, which are  time-inhomogeneous \Levy processes, often arising as useful examples in financial modeling. 
\end{abstract}

\noindent\keywords{association, orthant dependence,  Feller evolutions system, additive process,  time-inhomogeneous Markov process, comparison of Markov processes} 

\section{Introduction}

Feller processes are Markov processes which are useful models of dynamical systems that arise in finance and physics \cite{Bottcher2010}. Such processes are time-homogeneous and can be spatially inhomogeneous. A more general class of Markov processes are Feller evolution processes (FEP). General FEPs can be spatially inhomogeneous \textit{and} are time-inhomogeneous. Analogous to Feller processes, the Markov evolution of the FEP, called the Feller evolution system (FES), is strongly continuous. FEPs provide the structure for general temporally inhomogeneous models, such as additive processes and certain stochastic volatility models.

\par

Of a particular interest is the study of the dependence between the marginal processes and the dependence over time, which we call \textit{spatial dependence} and \textit{temporal dependence}, respectively. Some different notions of positive dependence include association (A), positive supermodular association (PSA), positive supermodular dependence (PSD), and positive orthant dependence (POD). One can better study the evolution of the process if the process satisfies a particular notion of spatial or temporal positive dependence. 

\par

It is known that (rich) Feller processes $X=(X_t)_{t\geq0}$ on $\R^d$ can be characterized by their state-space dependent characteristic triplet $(b(x), \Sigma(x), \nu(x,dy))$, where $b:\R^d\rightarrow\R^d$ describes the non-random behavior, $\Sigma:\R^d \rightarrow\R^{d\times d}$ describes the continuous diffusion-like behavior, and $\nu(\cdot,dy)$ is a measurable kernel describing the jump behavior of the process. Feller processes are often called L\'{e}vy-type processes, given that they behave locally like a \Levy process (for more on Feller processes, see \cite{Schilling2013}). 

\par

Feller evolution processes are Markov processes that are time-inhomogeneous and spatially inhomogeneous. Analogous to Feller processes, they also have a characteristic triplet $(b_t(x), \Sigma_t(x),\nu_t(x,dy))$, describing non-random, diffusion, and jump behavior, respectively, except that the characteristics are also time-dependent. In this paper, when a process with a characteristic triplet does not have a diffusion component $\Sigma$, we will refer to that process as a jump-process. 

\par

For \Levy processes $X$, the characterization of the positive dependence structures has been done by Herbst and Pitt  (1991) \cite{Herbst1991} in the case of Brownian motion with drift, i.e. $X\sim (b,\Sigma,0)$ and by Samorodnitsky (1995) \cite{Samorodnitsky1995} in the case of jump-\Levy processes, i.e. $X\sim (b,0,\nu)$. Samorodnitsky proved that jump-\Levy processes are spatially associated if and only if $\nu$ is concentrated on the positive and negative orthants $\R_+^d$ and $\R_-^d$, i.e. 
\begin{equation}
\label{resnick}
\nu((\R_+^d \cup \R_-^d)^c) = 0.
\end{equation}
This equivalence was also proven by Houdr\'{e} et. al. (1998) using covariance identities \cite{Houdre1998}. This result has since been extended to temporal association of jump-\Levy processes, and also to weaker forms of dependence in PSD and POD, by B\"{a}uerle (2008) \cite{Bauerle2008}. For general time-homogeneous Markov processes, characterizations for spatial and temporal association based on the generator have been given by Liggett (1985) \cite[p.80-83]{Liggett1985} and extended by  Szekli (1995) \cite[p.155]{Szekli1995} and R\"{u}schendorf  (2008)\cite{Ruschendorf2008}. Specifically, for Feller processes, Mu-fa Chen (1995) characterized spatial association for stochastically monotone diffusion-like processes, $(b(x),\Sigma(x),0)$, and Jie Ming Wang (2009) characterized spatial association for stochastically monotone jump-Feller processes,  $(b(x), 0, \nu(x,dy))$, under the condition
\begin{equation}
\label{resnickx}
\nu(x,(\R_+^d \cup \R_-^d)^c) = 0, \hspace{.4cm} \forall x\in\R^d
\end{equation}
 Tu (2019a)  extended Wang's results for association of jump-Feller processes with relaxed continuity and integrability conditions and also extended Liggett's characterization of association in \cite[p.80]{Liggett1985} from the generator to the extended generator, an integro-differential operator. Tu also extended the characterization to weaker positive dependence structures: WA, PSA, PSD, POD, PUOD, and PLOD \cite{Tu2018a}.

\par

As far as the author is concerned, little has been done on characterizing dependence structures for time-inhomogeneous Markov processes. The goal of this paper is provide such characterizations for FEPs, a general class of time-inhomogeneous Markov processes. This characterization will be based on the extended generators of the process. Moreover, we have interest in jump-FEPs, i.e. having characteristics $(b_t(x), 0, \nu_t(x,dy))$. We will provide a characterization of positive dependence based on the time-dependent \Levy measure: 
\begin{equation}
\label{resnicktx}
\nu_t(x,(\R_+^d \cup \R_-^d)^c) = 0, \hspace{.4cm}\forall t\geq0, \hspace{.1cm} \forall x\in\R^d.
\end{equation}

We will prove condition \eqref{resnicktx} is equivalent to spatial association, WA, PSA, PSD, POD, PUOD, and PLOD of the Feller evolution process. Our technique will be based on B\"{o}ttcher's transformation of FEP into a Feller process by adding another dimension to the process and space. From there, we can apply the results by Tu in \cite{Tu2018a} on Feller processes to study the spatial dependence of the FEP. Then we will provide examples to which we can apply the results, namely additive processes. Additionally, the techniques of our proofs can be used to prove comparison theorems of Feller evolution processes.

\par

It is important to clarify the distinction between this paper and a concurrent paper of ours, titled ``On the association and other forms of positive dependence for Feller processes" \cite{Tu2018a}. In the present paper, we characterize positive dependence for FEPs, which are more general than the time-homogeneous Feller processes studied in \cite{Tu2018a}. However, we need the results in \cite{Tu2018a} to prove the results in this present paper. More specifically, we utilize B\"{o}ttcher's transformation of time-inhomogeneous FEPs into time-homogeneous Feller processes (see \cite{Bottcher2013})  and, in a non-trivial way, apply our results in \cite{Tu2018a} to prove the positive dependence results of FEPs in this paper.  Additionally, the results here apply to a larger class of Markov processes, like additive processes. Hence, we felt that the difference in temporal behavior between FEPs and Feller processes, the necessity of \cite{Tu2018a} to prove results in this paper, and the difference in technique of proof are enough reason to have two distinct papers in this area.

\par

Additionally, we want to distinguish between the results here and the paper by R\"{u}schendorf et. al. \cite{Ruschendorf2016} due to some similarities. The authors of \cite{Ruschendorf2016} prove comparison theorems of  time-inhomogeneous Markov processes, including FEPs and processes with independent increments (PIIs).  These  theorems give conditions for comparing such Markov processes based on certain function classes induced by stochastic orderings. Some of the dependence structures we study in this paper are induced by the same stochastic orderings. However, we focus on characterizing dependence structures in FEPs.  Also, our style of proof differs, in that we make use of B\"{o}ttcher's homogeneous transformation of FEPs. Finally, we recognize that the comparison theorem we include in Section \ref{sec:comparison} is not as general as \cite[Thm.3.3]{Ruschendorf2016}, but it is interesting to show that B\"{o}ttcher's transformation is another nice technique to prove comparison theorems for time-inhomogeneous Markov processes. 

\par

Our paper is organized in the following way. In Section \ref{sec:background}, we give some background on the positive dependence structures, association, WA, PSA, PSD,  POD, PUOD, and PLOD. We also provide background on time-inhomogeneous Markov processes and Feller evolution processes. We will also summarize B\"{o}ttcher's transformation of FEP to a Feller process. In Section \ref{sec:mainresults}, we state and prove our main results about the positive dependence structures of stochastically monotone jump-FEPs. Finally, in Section \ref{sec:examples}, we provide applications to additive processes and comparison theorems.

\section{Background}
\label{sec:background}

\subsection{Positive dependence structures}

We first give a brief background on various positive dependence structures. For a more detailed description of these structures, see \cite{Tu2018a, TuThesis, Mueller2002}. Let $X = (X_1,...,X_d)$ be a random vector in $\R^d$. We say $X$ is \textbf{positively correlated (PC)} if $\Cov(X_i,X_j)\geq0$ for all $i,j\in\{1,...,d\}$. This is one of the weakest forms of positive dependence, and we are interested in stronger forms of positive dependence which will be of greater use in our study of stochastic processes. We list them here.

\par

Let $X=(X_1,...,X_d)$ be a random vector in $\R^d$.

\begin{definition}
{\rm 
$X$ is \textbf{associated (A)} if we have $$\Cov(f(X), g(X)) \geq0,$$
for all $f,g:\R^d\rightarrow\R$ non-decreasing in each component, such that $\Cov(f(X),g(X))$ exists.
}
\end{definition}


\begin{definition}
\label{def:WA}
{\rm 
$X$ is \textbf{weakly associated (WA)} if, for any pair of disjoint subsets $I,J\subseteq\{1,..,d\}$, with $|I| = k$, $|J|=n$,
\begin{equation*}
\label{WA}
\Cov(f(X_I), g(X_J))\geq0,
\end{equation*}
where $X_I := (X_i:i\in I)$, $X_J := (X_j:j\in J)$, for any  $f:\R^k\rightarrow\R$, $g:\R^{n}\rightarrow\R$ non-decreasing, such that  $\Cov(f(X_I), g(X_J))$ exists.
}
\end{definition}

\begin{definition}
{\rm 
$X$ is \textbf{positive supermodular associated (PSA)} if $\Cov(f(X), g(X))\geq0$ for all $f,g\in \mcl{F}_{ism} := \{h:\R^d\rightarrow\R, \text{ non-decreasing, supermodular}\}$. $f$ 
\textbf{Supermodular} means, for all $x,y\in\R^d$, $f(x\wedge y) + f(x\vee y) \geq f(x) + f(y),$
where $x\wedge y$ is the component-wise minimum, and $x\vee y$ is the component-wise maximum. 
}
\end{definition}

Now let $\hat{X} = (\hat{X}_1,...,\hat{X}_d)$ be a random vector such that for all $i$, $\hat{X}_i \stackrel{d}= X_i$ and $\hat{X}_i$'s are mutually independent.

\begin{definition}
\label{def:psd}
{\rm 
$X$ is \textbf{positive supermodular dependent (PSD)} if, for all $f:\R^d\rightarrow\R$ supermodular, 
$\E f(\hat{X}) \leq \E f(X)$.
}
\end{definition}

\begin{definition}
{\rm 
$X$ is \textbf{positive upper orthant dependent (PUOD)} if for all $t_1,...,t_d\in\R$,
\begin{align*}
\dP(X_1 >t_1,...,X_d >t_d) \geq \dP(X_1>t_1)...\dP(X_d>t_d).
\end{align*}
}
\end{definition}

\begin{definition}
{\rm 
$X$ is \textbf{positive lower orthant dependent (PLOD)} if for all $t_1,...,t_d\in\R$,
\begin{align*}
\dP(X_1 \leq t_1,...,X_d \leq t_d) \geq \dP(X_1\leq t_1)...\dP(X_d\leq t_d).
\end{align*}
}
\end{definition}

\begin{definition}
\label{def:pod}
{\rm 
$X$ is \textbf{positive orthant dependent (POD)} if  $X$ is PUOD and PLOD. 

\par

One can also state another equivalent definition to PUOD (PLOD). For $i=1,...,d$, let $f_i:\R\rightarrow\R_+$ be non-decreasing (non-increasing) functions.  Then $X=(X_1,...,X_d)$ \textbf{PUOD} (\textbf{PLOD}) if and only if 
\begin{center}
$\E \left(\prod_{i=1}^d f_i (X_i) \right) \geq \prod_{i=1}^d \E f_i (X_i).$
\end{center}
}
\end{definition}

Definitions \ref{def:psd}-\ref{def:pod} can also be stated in terms of stochastic orderings. For more on this, we refer the reader to M\"{u}ller and Stoyan's book \cite[Ch.3]{Mueller2002}. It is useful to see the relationship between these different forms of positive dependence. We state the relationships in Proposition \ref{propdepmap}.



\begin{proposition}
\label{propdepmap}
The  implications in Figure \ref{fig:prop2_1} hold.
\begin{center}
\begin{figure}[h]
  \centering
  \includegraphics[trim = {0cm 0cm 0cm 0cm}, scale=.75]{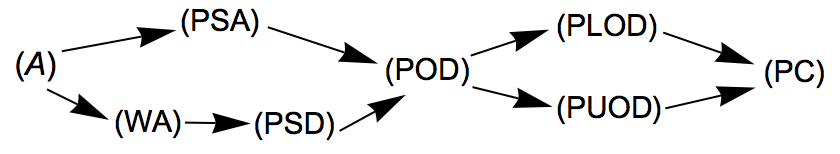}\\
  \caption{Implication map of various positive dependence structures}\label{fig:prop2_1}
\end{figure}
\end{center}
\end{proposition}

\begin{proof}
Proofs for these implications can be found in  \cite[Ch.3]{Mueller2002} and  \cite{TuThesis}. Implications involving PSD can be found in \cite{Christofides2004}. \end{proof}

These notions of dependence can be extended from random vectors to stochastic processes.  We will define, in the following subsection, dependence in a time-inhomogeneous Markov process.


\subsection{Time-inhomogeneous Markov processes}

Let $X = (X_t)_{t\geq0}$ be a Markov process in $\R^d$ on the space $(\Omega, \mcl{G}, (\mcl{G}_t)_{t\geq0}, \dP)$, where $\Omega$ is the sample space, $\mcl{G}$ is the $\sigma$-algebra, $(\mcl{G}_t)_{t\geq0}$ is the filtration, and $\dP$ is the probability measure. This means process $X$ satisfies the \textbf{Markov property}:
\begin{equation*}
\dP(X_t \in A|\mcl{G}_s) = \dP(X_t \in A|X_s), \hspace{.5cm} \forall \hspace{.1cm} s\leq t, \hspace{.1cm} A\in\mcl{B}(\R^d).
\end{equation*}
 We define the \textbf{Markov evolution} to be the family of linear operators $(T_{s,t})_{0\leq s\leq t<\infty}$ on $B_b(\R^d)$, the space of bounded functions from $\R^d$ to $\R$, by 
\begin{equation*}
T_{s,t} f(x) = \E f(X_t|X_s=x).
\end{equation*}
We will simply write $(T_{s,t})_{s\leq t}$ to mean $(T_{s,t})_{0\leq s\leq t<\infty}$ throughout this document.  We say Markov process $X$ is \textbf{normal} if $T_{s,t}:B_b(\R^d) \rightarrow B_b(\R^d)$ for all $0\leq s\leq t<\infty$. 

\begin{proposition}
\label{prop:propertiesME}
A normal Markov process $X$ with Markov evolution $(T_{s,t})_{s\leq t}$ satisfies the following properties:
\begin{enumerate}[noitemsep]
\item $T_{s,s}=I$ for all $s\geq0$.
\item $T_{r,s} T_{s,t} = T_{r,t}$, for all $0\leq r\leq s\leq t<\infty$ (Chapman-Kolmogorov).
\item $f\geq0$ implies $T_{s,t}f\geq0$ for all $0\leq s\leq t <\infty$ (positivity-preserving).
\item $||T_{s,t}||\leq 1$ for all $0\leq s\leq t<\infty$ (contraction).
\item $T_{s,t}1 = 1$.
\end{enumerate}
\end{proposition}

\begin{proof}
See Applebaum \cite[p.144]{Applebaum2004}.
\end{proof}

The time-inhomogeneous Markov processes that we will consider in this paper have Markov evolutions which satisfy a property called strong continuity. Consider the Banach space $(C_0(\R^d), ||\cdot||_\infty)$, where $C_0(\R^d)$ are functions from $\R^d$ to $\R$ that are continuous, bounded, and vanish at infinity, and $||\cdot||_\infty$ is the sup-norm. 
We say the Markov evolution is \textbf{strongly continuous} on $C_0(\R^d)$ if for every pair $0\leq s\leq t<\infty$
\begin{equation*}
\label{inhomo:strongcont}
\lim_{(u,v)\rightarrow(s,t)} ||T_{u,v} f - T_{s,t} f||_\infty = 0, \hspace{.5cm} \forall f\in C_0(\R^d).
\end{equation*}

If a Markov evolution $(T_{s,t})_{s\leq t}$ on $C_0(\R^d)$ is a strongly continuous, positivity-preserving, contraction satisfying Chapman-Kolmogorov equations, then we call $(T_{s,t})_{s\leq t}$ a \textbf{Feller evolution system (FES)} and its corresponding process $X=(X_t)_{t\geq0}$ a \textbf{Feller evolution process (FEP)}. The FES and FEP can be thought of as the time-inhomogeneous analogue to Feller semigroups and Feller processes (see B\"{o}ttcher et. al. \cite{Schilling2013} for background on Feller processes). 

\par

For a FES, we can define a family of left and right generators. The right-generators $(\mcl{A}_s^+)_{s\geq0}$ of FES $(T_{s,t})_{s\leq t}$ is defined by
\begin{equation*}
\mcl{A}_s^+ f = \lim_{h\searrow0} \frac{T_{s,s+h} f - f}{h}
\end{equation*} 
for all $f\in\mcl{D}(\mcl{A}_s^+)$, the subspace of functions in $C_0(\R^d)$ for which the above limit exists in $||\cdot||_\infty$. Similarly, the left-generators $(\mcl{A}_s^-, \mcl{D}(\mcl{A}_s^-))_{s\geq0}$ by 
\begin{equation*}
\mcl{A}_s^- f = \lim_{h\searrow0} \frac{T_{s-h,s} f - f}{h}.
\end{equation*} 
We can also express the left and right derivatives of the FES in terms of the left and right generators: 

\begin{enumerate}[itemsep=0.25pt]
\begin{multicols}{2}
\item $\displaystyle \frac{d}{dt}^+ T_{s,t}  = T_{s,t} \mcl{A}_t^+$ (forward eqn.)
\item $\displaystyle \frac{d}{dt}^- T_{s,t}   = T_{s,t} \mcl{A}_t^-$
\columnbreak 
\item $\displaystyle \frac{d}{ds}^+ T_{s,t}  = -\mcl{A}_s^+ T_{s,t}$
\item  $\displaystyle \frac{d}{ds}^- T_{s,t}  = -\mcl{A}_s^- T_{s,t}$ (backward eqn.).
\end{multicols}
\end{enumerate}

Assume now that $\mcl{D}(\mcl{A}_s^+),\mcl{D}(\mcl{A}_s^-) \supset C_c^\infty(\R^d)$, the space of smooth functions with compact supprt, for all $s\geq0$. By the theorem of Courr\`{e}ge \cite{Courrege1965}, we have that for every $s\geq0$, $-\mcl{A}_s^\pm|_{C_c^\infty(\R^d)}$ is a pseudo-differential operator:
\begin{equation}
\label{inhomo:pdo}
\mcl{A}_s^\pm|_{C_c^\infty(\R^d)} f(x) = (2\pi)^{-d/2} \int_{\R^d} e^{ix\cdot\xi} p_s^\pm(x,\xi) \hat{f}(\xi) d\xi
\end{equation}
where $-p_s^\pm(x,\xi)$ is a continuous negative definite function (cndf) for each $s\geq0$ (see \cite[Ch.4.5]{Jacob2001}). We call $p_s^\pm(x,\xi)$ the \textbf{symbol of the generator} $\mcl{A}_s^\pm$, and the $(p_s^\pm(x,\xi))_{s\geq0}$ the \textbf{family of symbols of the process}. When $C_c^\infty(\R^d)\subset \mcl{D}(\mcl{A}_s^+),\mcl{D}(\mcl{A}_s^-)$ for all $s\geq0$, we say that the generators have \textbf{rich domain} or that the associated \textbf{Markov process is rich}. In the FESs we study, the left and right generators will coincide. B\"{o}ttcher gives conditions for this situation \cite{Bottcher2013}, which we write in the following theorem.

\begin{theorem}[B\"{o}ttcher (2013) \cite{Bottcher2013}]
\label{leftrightequal}
Let $(T_{s,t})_{s\leq t}$ be a FES with left and right generators $(\mcl{A}_s^+, \mcl{D}(\mcl{A}_s^+))_{s\geq0}$ and $(\mcl{A}_s^-, \mcl{D}(\mcl{A}_s^-))_{s\geq0}$ with corresponding symbols $(p_s^+(x,\xi))_{s\geq0}$ and $(p_s^-(x,\xi))_{s\geq0}$. If 
\begin{equation}
\label{inhomocont}
p_s^\pm(x,\xi) \text{ is continuous in } s \text{ for all } x,\xi\in\R^d
\end{equation}
 and is bounded, i.e. there exists $C^\pm>0$ such that 
\begin{equation}
\label{inhomobounded}
p_s^\pm (x,\xi) \leq C^\pm (1+ |\xi|^2), \hspace{.5cm} \forall s\geq0,\hspace{.2cm} \forall x,\xi\in\R^d,
\end{equation}
then $\mcl{A}_s^+ = \mcl{A}_s^-$ for all $s\geq0$.
\end{theorem}

As a corollary to this theorem, conditions \eqref{inhomocont} and \eqref{inhomobounded} give us just one family of generators and symbols to consider: $(\mcl{A}_s)_{s\geq0}$ and $(p_s(x,\xi))_{s\geq0}$. Throughout this chapter, we often assume that $p_s(x,\xi)$ is \textbf{$s$-continuous} and \textbf{bounded}, i.e. satisfying \eqref{inhomocont} and \eqref{inhomobounded}, respectively.

\par

Assume now a rich domain, i.e. $C_c^\infty(\R^d)\subset \mcl{D}(\mcl{A}_s)$ for all $s\geq0$, and $p_s(x,\xi)$ is $s$-continuous and bounded. By 1-1 correspondence between cndfs and the L\'{e}vy-Khintchine formula, we have the following representation:
\begin{equation}
\label{eq:symbolLK}
p_s(x,\xi) = i b_s(x) \cdot \xi - \frac{1}{2} \xi \cdot \Sigma_s(x) \xi + \int_{y\neq0} (e^{i \xi\cdot y} - 1 - i \xi \cdot y\chi(y)) \nu_s(x,dy),
\end{equation}
where, for each $s\geq0$, $b_s:\R^d\rightarrow\R^d$ represents the non-random behavior, $\Sigma_s:\R^d \rightarrow\R^{d\times d}$ is a symmetric positive definite matrix which represents the continuous behavior, and $\nu_s(x,\cdot)$ is \Levy measure on $\R^d$ for all $x\in\R^d$ which represents the jump behavior. Function $\chi:\R^d\rightarrow\R$ is called the \textbf{cut-off function}. Unless, we specify otherwise, in this paper, we set $\chi(y):=\mathbbm{1}_{(0,1)} (|y|)$. We call $(b_s(x), \Sigma_s(x), \nu_s(x,dy))$ the \textbf{(L\'{e}vy) characteristic triplet} of process $X$. We have for each $s\geq0$, an \textbf{integro-differential operator} $I(p_s)$ defined on $C_b^2(\R^d)$ by substituting the L\'{e}vy-Khintchine form in equation \eqref{eq:symbolLK} into \eqref{inhomo:pdo}, and, by elementary Fourier analysis, 
\begin{equation}
\label{inhomo:ido}
I(p_s) f(x) = b_s(x) \cdot \nabla f(x) + \frac{1}{2} \nabla \cdot \Sigma_s(x) \nabla f(x) + \int_{y\neq0} (f(x+y) - f(x) - \nabla f(x)\cdot y\chi(y))\nu_s(x,dy).
\end{equation}
$I(p_s)$ clearly extends $\mcl{A}_s$ onto $C_b^2(\R^d)$, i.e. $I(p_s)|_{\mcl{D}(\mcl{A}_s)} = \mcl{A}_s$. 

\par
Now, we wish to define what dependence means in these processes.

\subsubsection{Dependence, monotonicity in time-inhomogeneous Markov processes}

Let $C_b(\R^d)$ be the space of continuous, bounded functions, and let $\mcl{F}_i$ be the space of functions from $\R^d$ to $\R$ that are non-decreasing componentwise. (Note: we often don't specify the dimension of the domain of the functions in $\mcl{F}_i$. This is because we often intersect this space $\mcl{F}_i$ with other spaces in which we do specify the domain. For example, $C_b(\R^d)\cap \mcl{F}_i$ would mean that $\mcl{F}_i$ are non-decreasing functions on $\R^d$, whereas $C_b(\R^n)\cap \mcl{F}_i$ would mean that $\mcl{F}_i$ are non-decreasing functions on $\R^n$.)

\begin{definition}
\label{def:spatialassocinhomo}
{\rm Let $X = (X_t)_{t\geq0}$ be a time-inhomogeneous Markov process with Markov evolution $(T_{s,t})_{s\leq t}$. We say $X$ is \textbf{spatially associated} if for all $s\leq t$, $f,g\in C_b(\R^d) \cap \mcl{F}_i$, we have $T_{s,t} fg \geq T_{s,t} f \hspace{.1cm} T_{s,t} g.$
}
\end{definition}

\begin{definition}
\label{def:tempassocinhomo}
{\rm 
Let $X = (X_t)_{t\geq0}$ be a time-inhomogeneous Markov process. We say $X$ is \textbf{temporally associated} if for all $0\leq t_1<...<t_n$, $(X_{t_1},...,X_{t_n})$ is associated in $\R^{dn}$.
}
\end{definition}

\begin{remark}
\label{rem:spatassoc}
{\rm
\begin{enumerate}[noitemsep,label = (\roman*)]
\item The meaning of Definition \ref{def:spatialassocinhomo} can be interpreted as the following. For each $x\in\R^d$ and $s\leq t$, $X_t$ is an associated random vector conditioned on the event $\{X_s=x\}$, i.e. $\E [f(X_t)g(X_t)|X_s=x] \geq \E[f(X_t)|X_s=x] \cdot \E [g(X_t)|X_s=x].$
Such a definition is more useful in applications. For example, see \cite{Lindqvist1987} for an application in reliability theory. 

\item We can define other forms of positive dependence in time-inhomogeneous Markov processes if we replace ``associated" in Remark \ref{rem:spatassoc}(i) with ``WA," ``PSA," ``PSD," ``POD," ``PUOD," or ``PLOD." 

\item Our focus in this paper will be on spatial dependence. Lindqvist (1987) in \cite{Lindqvist1987} refers to light conditions which make Definition \ref{def:spatialassocinhomo} imply Definition \ref{def:tempassocinhomo}. We refer the reader to that paper.

\end{enumerate}
}
\end{remark}

Our interest will lie in Feller evolution processes which are stochastically monotone. For a general time-inhomogeneous Markov process, this is defined in the following way:

\begin{definition}
\label{def:stochmonoinhomo}
{\rm Let $X = (X_t)_{t\geq0}$ be a time-inhomogeneous Markov process with Markov evolution $(T_{s,t})_{s\leq t}$. We say $X$ is \textbf{stochastically monotone} if for all $s\leq t$, $f\in C_b(\R^d) \cap \mcl{F}_i$, we have $T_{s,t} f\in \mcl{F}_i$.
}
\end{definition}

There are few results in the literature, as far as the author can tell, regarding dependence structures in time-inhomogeneous Markov processes. There are, however, several useful results in the characterization of dependence structures in time-homogeneous Feller processes, most notably Mu-fa Chen (1993) \cite{Chen1993}, Jie Ming Wang (2009) \cite{Wang2009}, and Tu (2019a) \cite{Tu2018a}. Thus, to characterize positive dependence structures in Feller evolution processes, we can transform the time-inhomogeneous FEP into a time-homogeneous Feller process and apply results on Feller processes to answer questions about the FEP! We do this transformation following the prescription given by B\"{o}ttcher \cite{Bottcher2013}, and then use results from Tu  \cite{Tu2018a} to characterize the dependence structures in FEPs. We give an overview of B\"{o}ttcher's transformation in the following subsection and highlight some important results from his paper \cite{Bottcher2013}.




\subsection{Time-homogeneous transformation of a time-inhomogeneous Markov process}
\label{transformation}

For the sake of brevity, we will omit background on Feller processes and general time-homogeneous Markov processes. If the reader would like more background information on those topics, please see \cite{Tu2018a} or \cite{Schilling2013}.

\par

Time-homogeneous Markov processes have very nice properties and analytical tools. To take advantage of those tools in the time-inhomogeneous case, we can transform our time-inhomogeneous process $X$ into a time-homogeneous process $\tilde{X}$ by adding another (deterministic) component to the process. We will outline the transformation of $X$ to $\tilde{X}$ in this subsection. We follow the prescription used in B\"{o}ttcher \cite{Bottcher2013}. 

\par

Let $X=(X_t)_{t\geq0}$ be a time-inhomogeneous Markov process with sample space \\$(\Omega, \mcl{G}, (\mcl{G}_t)_{t\geq0}, \dP)$, state space $(\R^d, \mcl{B}(\R^d))$, and Markov evolution $(T_{s,t})_{s\leq t}$, and corresponding Markov kernels $(P_{s,t})_{s\leq t}$ defined by 
\begin{equation*}
P_{s,t}(x,A) := T_{s,t}\mathbbm{1}_A(x).
\end{equation*}

We define a transformed process in the following manner.

\begin{prescription}
\label{pre1}
{\rm To define the \underline{new sample space}, let $\tilde{\Omega} := \R_+ \times \Omega$, where elements $\tilde{\omega} = (s,\omega)$, with $s\geq0$, $\omega\in\Omega$. 
The $\sigma$-algebra will be $\tilde{\mcl{G}}$, defined by 
\begin{equation*}
\tilde{\mcl{G}} = \{A\subset \tilde{\Omega}: A_s \in \mcl{G}, \hspace{.2cm} \forall s\geq0\}, 
\end{equation*}
where $A_s:=\{\omega\in\Omega: (s,\omega)\in A\}$. The \underline{new state space} will be defined to be $\R_+\times\R^d$ with $\sigma$-algebra $\tilde{\mcl{B}}$ defined by 
\begin{equation*}
\tilde{\mcl{B}} = \{B\subset \R_+\times\R^d: B_s \in \mcl{B}(R^d), \hspace{.2cm} \forall s\geq0\},
\end{equation*}
where $B_s:=\{x\in\R^d: (s,x)\in B\}$. From this, we can define a \underline{new process} $\tilde{X} = (\tilde{X}_t)_{t\geq0}$ on $\R_+\times \R^d$ by the prescription
\begin{equation*}
\tilde{X}_t(\tilde{\omega}) = (s+t, X_{s+t}(\omega)),
\end{equation*}
where $\tilde{\omega} = (s,\omega)$. The family of  probability measures $(\tilde{\dP}^{\tilde{x}})_{\tilde{x}\in\R_+\times\R^d}$ is given by 
\begin{equation*}
\tilde{\dP}^{\tilde{x}}(A|\tilde{X}_0 = \tilde{x})= \tilde{\dP}^{\tilde{x}}(A|\tilde{X}_0 = (s,x)) = \dP(A_s|X_s=x), \hspace{.5cm} A\in\tilde{\mcl{G}}.
\end{equation*}
From this we can define the transition kernel $(\tilde{P}_t)_{t\geq0}$  by 
\begin{equation}
\label{transformedkernel}
\tilde{P}_t(\tilde{x}, B) : = \tilde{\dP}^{\tilde{x}}(\tilde{X}_t\in B| \tilde{X}_0 = \tilde{x}) = \dP(X_{s+t} \in B_{s+t}|X_s=x), \hspace{.5cm} B\in \tilde{\mcl{B}}.
\end{equation}
Thus, this prescription has given us a process $\tilde{X}=(\tilde{X}_t)_{t\geq0}$ with sample space \\ $(\tilde{\Omega}, \tilde{\mcl{G}},  \tilde{\dP}^{\tilde{x}})_{\tilde{x}\in\R_+\times \R^d}$, where $\tilde{x}$ represents the starting point of process $\tilde{X}$, i.e. $\tilde{\dP}^{\tilde{x}} (\tilde{X}_0 =\tilde{x})=1$, and state space $(\R_+ \times \R^d, \tilde{\mcl{B}})$. 
}
 \end{prescription}
 
 \par
 
 The process $\tilde{X}$ is a time-homogeneous Markov process, with transition semigroup  $(\tilde{T}_t)_{t\geq0}$ on $(B_b(\R_+\times\R^d), ||\cdot||_\infty)$, given by
 \begin{equation}
 \label{eq:semigroup}
\tilde{T}_t f(\tilde{x}) = \tilde{\E}^{\tilde{x}} f(\tilde{X}_t) = \E(f_{s+t}(X_{s+t}) | X_s=x) = T_{s,s+t} f_{s+t}(x),
\end{equation}
where $f_{s+t}:\R^d\rightarrow\R$ is defined by $f_{s+t}(x): = f(s+t,x)$ (See B\"ottcher \cite{Bottcher2013}). 
  
 


\par 
When given a time-inhomogeneous Markov process $X$ on sample space $(\Omega, \mcl{G},(\mcl{G}_t)_{t\geq0},\dP)$ and state space $(\R^d, \mcl{B}(\R^d))$, we call the  process $\tilde{X} = (\tilde{X}_t)_{t\geq0}$ on  sample space \\$(\tilde{\Omega}, \tilde{\mcl{G}},(\tilde{\mcl{G}}_t)_{t\geq0}, \tilde{\dP}^{\tilde{x}})_{\tilde{x}\in\R_+\times \R^d}$, state space $(\R_+ \times \R^d, \tilde{\mcl{B}})$, and semigroup $(\tilde{T}_t)_{t\geq0}$ given by Prescription \ref{pre1}  the \textbf{transformed process of $X$}. 


This transformed process $\tilde{X}$ has many nice properties and representations.  If $X$ is a rich FEP on $\R^d$ with FES $(T_{s,t})_{s\leq t}$, generators $(\mcl{A}_s, \mcl{D}(\mcl{A}_s))_{s\geq0}$, bounded and $s$-continuous symbols $p_s(x,\xi)$, characteristic triplets $(b_s(x), \Sigma_s(x), \nu_s(x,dy)$, extended generators $(I(p_s), C_b^2(\R^d))_{s\geq0}$,  we have that $\tilde{X}$ is a Feller process with  Feller semigroup  $(\tilde{T}_t)_{t\geq0}$, generator $(\tilde{\mcl{A}}, \mcl{D}(\tilde{\mcl{A}}))$, symbol $\tilde{p}(\tilde{x},\tilde{\xi})$, characteristic triplet $(\tilde{b}(\tilde{x}), \tilde{\Sigma}(\tilde{x}), \tilde{\nu}(\tilde{x}, d\tilde{y}))$, and extended generator (integro-differential operator) $(I(\tilde{p}), C_b^2(\R_+ \times \R^d))$. These objects have the following representations:

$\tilde{b}:\R^{d+1}\rightarrow \R^{d+1}$ defined by
\begin{equation}
\label{eq:b}
\tilde{b}(\tilde{x})  = \tilde{b}(s,x)  = (1,b_s(x)),
\end{equation}
$\tilde{\Sigma}:\R^{d+1}\rightarrow \R^{d+1\times d+1}$ defined by 
\begin{equation}
\begin{split}
\tilde{\Sigma}^{i0}(\tilde{x}) & = 0, \hspace{.2cm} \forall j=0,...,d
\\ \tilde{\Sigma}^{0j}(\tilde{x}) & = 0, \hspace{.2cm} \forall i=0,...,d
\\ \tilde{\Sigma}^{ij}(\tilde{x}) & = \Sigma_s^{ij}(x) \hspace{.2cm} \forall i,j=1,...,d
\end{split}
\end{equation}
and $\tilde{\nu}(\tilde{x}, d\tilde{y})$ is a \Levy measure on $\mcl{B}(\R^{d+1}\setminus\{0\})$ given by  
\begin{equation}
\tilde{\nu}(\tilde{x},d\tilde{y}) = \nu_s(x,dy)\delta_0(dr),
\end{equation} 
where $\tilde{y} = (r,y)$, and $\delta_0$ is Dirac measure at $0$. 

Symbol $\tilde{p}(\tilde{x},\tilde{\xi}): \R^{d+1}\times \R^{d+1} \rightarrow\C$ is given by 
\begin{equation}
\tilde{p} (\tilde{x}, \tilde{\xi})  = i\tilde{b}(\tilde{x}) \cdot \tilde{\xi} -  \frac{1}{2} \tilde{\xi} \cdot \tilde{\Sigma}(\tilde{x})\tilde{\xi}
 + \int_{\tilde{y}\neq0} (e^{i\tilde{\xi}\cdot\tilde{y}} - 1 - i \tilde{\xi}\cdot \tilde{y}\chi(\tilde{y}))\tilde{\nu}(\tilde{x},d\tilde{y})
\end{equation}
or 
\begin{equation}
\tilde{p} (\tilde{x}, \tilde{\xi}) = ir + p_s(x,\xi), \hspace{.5cm} \tilde{x} = (s,x), \hspace{.2cm} \tilde{\xi} = (r,\xi).
\end{equation}

Let $f\in C_b^2(\R_+\times\R^d)$, where $f=f(\tilde{x}) = f(s,x)$. Define $f_s(x) := f(s,x) \in C_b^2(\R^d)$ (where $s$ is fixed). Extended generator $I(\tilde{p})$ is an extension of $\mcl{A}$, i.e. $I(\tilde{p})|_{\mcl{D}(\mcl{\tilde{A}})} = \mcl{\tilde{A}}$, and is given by 
\begin{equation}
\begin{split}
I(\tilde{p}) f(\tilde{x}) &= \tilde{b}(\tilde{x}) \cdot \nabla f(\tilde{x}) +  \frac{1}{2} \nabla \cdot \tilde{\Sigma}(\tilde{x}) \nabla f(\tilde{x})
 + \int_{\tilde{y}\neq0} (f(\tilde{x} + \tilde{y}) - f(\tilde{x}) - \nabla f(\tilde{x})\cdot \tilde{y}\chi(\tilde{y}))\tilde{\nu}(\tilde{x},d\tilde{y})
\end{split}
\end{equation} 
or 
\begin{equation}
I(\tilde{p}) f(\tilde{x}) = \frac{\partial}{\partial s} f(s,x) + I(p_s) f_s(x).
\end{equation}

An additional nice property of the symbol $\tilde{p}(\tilde{x},\tilde{\xi})$ is that if $C_c^\infty(\R^d)$ is a core of $\mcl{A}_s$, then $\tilde{p}(\tilde{x},\tilde{\xi})$ is a bounded symbol, i.e.  there exists $C>0$ such that
\begin{equation}
\label{eq:bddsymb}
\sup_{\tilde{x}\in\R_+\times\R^d} |\tilde{p}(\tilde{x},\tilde{\xi})| \leq C(1+|\tilde{\xi}|^2), \hspace{.3cm} \text{ for all } \hspace{.1cm}\tilde{\xi}\in\R_+\times\R^d.
\end{equation}
For proofs and more details of this property and these formulas, see B\"{o}ttcher \cite[Thm. 3.2, 3.3, Cor. 3.5, Lem. 3.7]{Bottcher2013} and Tu \cite[Ch. 4]{TuThesis}

\section{Main results}
\label{sec:mainresults}

\subsection{Association of FEPs}

We give a characterization of spatial association for Feller evolution processes based on the extended generators $I(p_s)$. We apply this to characterize spatial association of such processes of the jump variety, i.e. $(b_s(x), 0, \nu_s(x,dy))$. These results are given in Theorems \ref{thm:inhomoliggett} and \ref{thm:inhomoassocjump}. We first need the following useful lemmas from \cite{Tu2018a} about (time-homogeneous) Feller processes.

\begin{lemma}[Theorem 3.2 of Tu  (2019a) \cite{Tu2018a}]
\label{extliggett}
Let $Y=(Y_t)_{t\geq0}$ be a Feller processes in $\R^n$ (with rich domain) with  a stochastically monotone transition semigroup $(T_t)_{t\geq0}$, a generator $(\mcl{A},\mcl{D}(\mcl{A}))$, bounded symbol $p(x,\xi)$, and an integro-differential operator $I(p)$. Assume $x\mapsto p(x,0)$ is continuous. Then 
\begin{equation*}
\label{eq:extliggett}
I(p) fg \geq f I(p) g + g I(p) f, \hspace{.5cm} \forall f,g\in C_b^2(\R^n)\cap \mcl{F}_i
\end{equation*}
\noindent if and only if 
\begin{equation*}
\label{semiass}
\forall t\geq0,\hspace{.5cm} T_t fg \geq T_t f \cdot  T_t g,  \hspace{.5cm} \forall f,g\in C_b(\R^n)\cap\mcl{F}_i.
\end{equation*}
\end{lemma}

\begin{lemma}[Lemma 3.3 of Tu (2019a) \cite{Tu2018a}]
\label{extcauchy}
Let $(\mcl{A},\mcl{D}(\mcl{A}))$ be a (rich) Feller generator of a Feller semigroup $(T_t)_{t\geq0}$ with bounded symbol $p(x,\xi)$ satisfying $x\mapsto p(x,0)$ continuous Let $I(p)$ be the extended generator on $C_b^2(\R^n)$. Suppose $F,G:[0,\infty)\rightarrow C_b(\R^n)$ such that 

\mypar

\noindent (a) $F(t)\in\mcl{D}(I(p))$ for all $t\geq0$
\\
(b) $G(t)$ is continuous on $[0,\infty)$ (locally uniformly)
\\ 
(c) $F'(t) = I(p) F(t) + G(t)$ for all $t\geq0$.
\mypar

\noindent Then $\displaystyle F(t) = T_t F(0) + \int_0^t T_{t-s} G(s) ds.$
\end{lemma}

\begin{lemma}[Theorem 3.3 of  Tu (2019a) \cite{Tu2018a}]
\label{PODnec}
Let $X = (X_t)_{t\geq0}$ be a rich Feller process in $\R^d$ with symbol $p(x,\xi)$ and triplet $(b(x), 0, \nu(x,dy))$. Then $X_t$ is  PUOD for each $t\geq0$ implies condition (\ref{resnickx}): $$\nu(x,(\R_+^d\cup \R_-^d)^c)=0, \hspace{.5cm} \forall x\in\R^d.$$
\end{lemma}

\begin{remark}
{\rm 
Observe that in Lemma \ref{PODnec}, we did not assume stochastic monotonicity, since we do not need that assumption for the proof (see \cite[Thm. 3.3]{Tu2018a}). But in order for condition (\ref{resnickx}) to be \textit{equivalent} to spatial PUOD, then we need the assumption of stochastic monotonicity. To understand why, see our paper \cite[Thm. 3.1]{Tu2018a}, which shows that under stochastic monotonicity, a jump-Feller process is associated if and only if condition (\ref{resnickx}) is satisfied. Hence, by Proposition \ref{propdepmap}, which says association implies PUOD (and all other dependence structures mentioned in this paper), and Lemma \ref{PODnec}, we have that a stochastically monotone jump-Feller process is PUOD if and only if condition (\ref{resnickx}) is satisfied (see \cite[Cor. 3.1]{Tu2018a}).
}
\end{remark}

Now we state and prove the main theorems of this paper, which can found in Theorems \ref{thm:inhomoliggett}, \ref{thm:inhomoassocjump}, \ref{thm:inhomoPUODnec}.

\begin{theorem}
\label{thm:inhomoliggett}
Let $X=(X_t)_{t\geq0}$ be a Feller evolution process with Feller evolution system $(T_{s,t})_{s\leq t}$, generators $(\mcl{A}_s)_{s\geq0}$ with rich domains, and that $C_c^\infty(\R^d)$ is the core for $\mcl{A}_s$, for all $s\geq0$. Let the corresponding symbols $p_s(x,\xi)$ be $s$-continuous and bounded, 
and $I(p_s)$ be the integro-differential operator (extended generator) of $X$. If $X$ is stochastically monotone, then $X$ is spatially associated if and only if
\begin{equation*}
I(p_s) fg \geq  f I(p_s) g + g I(p_s) f, \hspace{.5cm} \forall f,g\in C_b^2(\R^d)\cap\mcl{F}_i, \hspace{.1cm} s\geq0.
\end{equation*}
\end{theorem}

\begin{proof}
Let $\tilde{X} = (\tilde{X}_t)_{t\geq0}$ on $\R_+\times \R^d$ be transformation of $X$, given by Prescription \ref{pre1}, which has Feller semigroup $(\tilde{T}_t)_{t\geq0}$, generator $(\mcl{\tilde{A}},\mcl{D}(\mcl{\tilde{A}}))$ with rich domain, bounded symbol $\tilde{p}(\tilde{x},\tilde{\xi})$, characteristics $(\tilde{b}(\tilde{x}), \tilde{\Sigma}(\tilde{x}), \tilde{\nu}(\tilde{x},d\tilde{y}))$ and extended generator $I(\tilde{p})$ on $C_b^2(\R_+\times\R^d)$, as given to us by equations \eqref{eq:semigroup} to \eqref{eq:bddsymb}.

\mypar

\noindent ($\Rightarrow$). Assume $T_{s,s+t} fg \geq T_{s,s+t} f T_{s,s+t} g$ for all $s,t\geq0$ and all $f,g\in C_b(\R^d)\cap\mcl{F}_i$. Choose $h,k\in C_b(\R_+\times \R^d)\cap \mcl{F}_i$, $\tilde{x} = (s,x)$. Then $h_{s+t},k_{s+t}\in C_b(\R^d)\cap\mcl{F}_i$, and 
\begin{align*}
\tilde{T}_t hk(\tilde{x}) & = T_{s,s+t} h_{s+t} k_{s+t} (x)
 \geq T_{s,s+t} h_{s+t}(x) \cdot T_{s,s+t} k_{s+t} (x)
 = \tilde{T}_t h(\tilde{x}) \cdot \tilde{T}_t k(\tilde{x}). 
\end{align*}
Observe that the bounded symbol $\tilde{p}(\tilde{x},\tilde{\xi})$  also satisfies $\tilde{x}\mapsto \tilde{p}(\tilde{x}, 0)$ is continuous, since 
\begin{align*}
\tilde{p}(\tilde{x},0) = i\cdot(0) + p_s(x,0) = 0.
\end{align*}
So by Lemma \ref{extliggett}, we have that the extended generator $I(\tilde{p})$ satisfies 
\begin{equation}
\label{blehbleh}
I(\tilde{p}) hk \geq h I(\tilde{p}) k + k I(\tilde{p}) h, \hspace{.5cm} h,k\in C_b^2(\R_+\times\R^d)\cap \mcl{F}_i.
\end{equation}
Choose $f,g\in C_b(\R^d)\cap\mcl{F}_i$. Then there exists $h,k\in C_b(\R_+\times \R^d)\cap \mcl{F}_i$, where $h,k$ are constant with respect to the first argument, and $f(x) = h(\tilde{x})$ and $g(x) = k(\tilde{x})$. Choose $\tilde{x} = (s,x)$. Then
\begin{align*}
I(\tilde{p}) hk(\tilde{x}) & = \frac{\partial}{\partial s} h(s,x)k(s,x) + I(p_s) h_s k_s(x)
 = 0 + I(p_s) fg(x)
 =I(p_s) fg(x)
\end{align*}
and 
\begin{align*}
& h(\tilde{x}) I(\tilde{p}) k(\tilde{x}) + k(\tilde{x}) I(\tilde{p}) h(\tilde{x})
\\  & = h(s,x) \left( \frac{\partial}{\partial s} k(s,x) + I(p_s) k_s(x) \right) + k(s,x)  \left(\frac{\partial}{\partial s} h(s,x) + I(p_s) h_s(x)\right)
\\ & = h(\tilde{x}) I(p_s) k_s(x) + k(\tilde{x}) I(p_s) h_s(x)
\\ & = f(x) I(p_s) g(x) + g(x) I(p_s) f(x).
\end{align*}
Thus, by \eqref{blehbleh}, we have $I(p_s) fg \geq f I(p_s) g + g I(p_s) f.$

\mypar

\noindent ($\Leftarrow$). Assume, for all $s\geq0$, $I(p_s) fg \geq f I(p_s) g + g I(p_s) f$, $\forall f,g \in C_b^2(\R^d)\cap\mcl{F}_i$. Choose $f,g\in C_b^2(\R_+\times \R^d)\cap\mcl{F}_i$, $\tilde{x}=(s,x)$, then
\begin{equation}
\label{proofsuffliggett}
\begin{split}
I(\tilde{p}) fg(\tilde{x}) & = \frac{\partial}{\partial s} f(s,x)g(s,x) + I(p_s) f_s g_s(x)
\\ & = f(s,x) \frac{\partial}{\partial s}g(s,x) + g(s,x)  \frac{\partial}{\partial s} f(s,x) + I(p_s) f_s g_s(x)
\\ & \geq  f(s,x) \frac{\partial}{\partial s}g(s,x) + g(s,x)  \frac{\partial}{\partial s} f(s,x) +   f_s(x) I(p_s) g_s(x) + g_s(x) I(p_s) f_s(x)
\\ & = f(s,x)\left( \frac{\partial}{\partial s}g(s,x)  + I(p_s) g_s(x) \right) + g(s,x) \left(\frac{\partial}{\partial s}f(s,x)  + I(p_s) f_s(x)\right)
\\ & = f(\tilde{x}) I(\tilde{p}) g(\tilde{x}) + g(\tilde{x}) I(\tilde{p}) f(\tilde{x}).
\end{split}
\end{equation}
Note that we assumed $(T_{s,t})_{s\leq t}$ is stochastically monotone. However, this does not imply that $(\tilde{T}_t)_{t\geq0}$ is stochastcally monotone. To see this, choose $\tilde{x}=(s,x)$ and $\tilde{y}= (r,y)$, where $\tilde{x}\leq \tilde{y}$ with $s<r$. Then let $f\in\mcl{F}_i\cap C_b(\R_+\times \R^d)$. Observe that 
\begin{align*}
\tilde{T}_t f(\tilde{x}) & = \tilde{\E}^{\tilde{x}} f(\tilde{X}_t)  
 = \E (f_{s+t} (X_{s+t})|X_s=x)
 \not\leq  \E (f_{r+t} (X_{r+t})|X_r=y)
 = \tilde{T}_t f(\tilde{y})
\end{align*}
since the sample paths of $X$ may not be monotone non-decreasing. But we can still get our desired result from the stochastic monotonicity of $(T_{s,t})_{s\leq t}$. Fix $s\geq0$. Choose $h,k\in C_b^2(\R_+\times\R^d)\cap \mcl{F}_i$. Then $\tilde{T}_t h|_{\{s\}\times\R^d}$, $\tilde{T}_t k|_{\{s\}\times\R^d}$, $\tilde{T}_t hk|_{\{s\}\times\R^d} \in C_b^2(\{s\}\times\R^d)\cap \mcl{F}_i$ for a fixed $s\geq0$. It is easy to see that these functions will be in $C_b^2(\{s\}\times\R^d)$. To see that they are non-decreasing on $\{s\}\times\R^d$, choose $\tilde{x}: =(s,x)\leq (s,y) =:\tilde{y}$. Then
\begin{align*}
\tilde{T}_t h|_{\{s\}\times\R^d} (\tilde{x})  =\tilde{T}_t h|_{\{s\}\times\R^d} (s,x) & = T_{s,s+t} h_{s+t} (x)
 \leq T_{s,s+t} h_{s+t}(y)
 = \tilde{T}_t h|_{\{s\}\times\R^d} (\tilde{y})
\end{align*}
by stochastic monotonicity of $(T_{r,t})_{r\leq t}$. Observe that there exists $v \in C_b^2(\R_+\times\R^d)\cap\mcl{F}_i$ such that $v$ is constant with respect to the first argument in $\R_+$ and $v(s,x) = \tilde{T}_t h|_{\{s\}\times\R^d} (s,x)$. Similarly, there is $w\in  C_b^2(\R_+\times\R^d)\cap\mcl{F}_i$ such that $w$ is constant with respect to first argument, and $w(s,x) = \tilde{T}_t k|_{\{s\}\times\R^d} (s,x)$. By inequality \eqref{proofsuffliggett}, we have $$I(\tilde{p}) vw \geq v I(\tilde{p}) w + w I(\tilde{p})v,$$
which implies for any $x\in\R^d$, with $\tilde{x} = (s,x)$, 
\begin{equation}
\label{proofsuffliggett1}
\begin{split}
I(\tilde{p}) \left(\tilde{T}_t h|_{\{s\}\times\R^d} \tilde{T}_t k|_{\{s\}\times\R^d}\right) (\tilde{x}) & \geq \tilde{T}_t h|_{\{s\}\times\R^d} (\tilde{x}) \cdot  I(\tilde{p}) \tilde{T}_t k|_{\{s\}\times\R^d} (\tilde{x})  
\\ & \hspace{.3cm} + \tilde{T}_t k|_{\{s\}\times\R^d} (\tilde{x}) \cdot I(\tilde{p}) \tilde{T}_t h|_{\{s\}\times\R^d} (\tilde{x})
\end{split}
\end{equation}

\noindent Now define $F, G:[0,\infty)\rightarrow C_b(\R_+\times\R^d)$, by 
\begin{equation*}
F(t): = \tilde{T}_t hk - \tilde{T}_t h \cdot \tilde{T}_t k \hspace{.5cm} \text{ and } \hspace{.5cm} G(t):= F'(t) - I(\tilde{p})F(t).
\end{equation*}
It is not hard to verify that $F,G$ are continuous on $[0,\infty)$ with respect to local uniform convergence. By Lemma \ref{extcauchy}, we have the solution
\begin{equation*}
F(t) = \tilde{T}_t F(0) + \int_0^t \tilde{T}_{t-r} G(r) dr = \int_0^t \tilde{T}_{t-r} G(r) dr.
\end{equation*} 
Now, choose $\tilde{x} = (s,x)$. Then by \eqref{proofsuffliggett1}
\begin{align*}
F'(t)(\tilde{x}) & = I(\tilde{p}) \tilde{T}_t hk (\tilde{x}) - (\tilde{T}_t h(\tilde{x}) \cdot I(\tilde{p}) \tilde{T}_t k(\tilde{x}) + \tilde{T}_t k(\tilde{x}) \cdot I(\tilde{p}) \tilde{T}_t h(\tilde{x}))
\\ & = I(\tilde{p}) \tilde{T}_t hk|_{\{s\}\times\R^d} (\tilde{x})  - \left(\tilde{T}_t h|_{\{s\}\times\R^d} (\tilde{x}) \cdot  I(\tilde{p}) \tilde{T}_t k|_{\{s\}\times\R^d} (\tilde{x})  \right.
\\ & \left.\hspace{3.85cm} + \tilde{T}_t k|_{\{s\}\times\R^d} (\tilde{x}) \cdot I(\tilde{p}) \tilde{T}_t h|_{\{s\}\times\R^d}(\tilde{x}) \right)
\\ & \geq  I(\tilde{p}) \tilde{T}_t hk|_{\{s\}\times\R^d} (\tilde{x}) -  I(\tilde{p}) \left(\tilde{T}_t h|_{\{s\}\times\R^d} \tilde{T}_t k|_{\{s\}\times\R^d}\right) (\tilde{x})
\\ & = I(\tilde{p}) F(t)|_{\{s\}\times\R^d} (\tilde{x})
\\ & =  I(\tilde{p}) F(t) (\tilde{x}).
\end{align*}

\noindent Thus, $G(t)(\tilde{x}) = F'(t)(\tilde{x})  -  I(\tilde{p}) F(t) (\tilde{x}) \geq0$. In other words, $G(t)|_{\{s\}\times\R^d} \geq0$. Hence, 
\begin{equation*}
F(t)|_{\{s\}\times\R^d} = \int_0^t \tilde{T}_{t-r} G(r)|_{\{s\}\times\R^d} dr \geq0.
\end{equation*}
This finally yields $\tilde{T}_t hk(\tilde{x}) \geq \tilde{T}_t h(\tilde{x}) \cdot \tilde{T}_t k(\tilde{x})$, for all $\tilde{x}\in \{s\}\times\R^d$, which then yields 
\begin{equation}
\label{proofsuffassoc}
T_{s,s+t} h_{s+t} k_{s+t} (x) \geq T_{s,s+t} h_{s+t} (x)\cdot T_{s,s+t} k_{s+t} (x)
\end{equation}
for all $x\in\R^d$. Now let $f,g\in C_b^2(\R^d)\cap \mcl{F}_i$. Then there are functions $h,k\in C_b^2(\R_+\times\R^d)\cap\mcl{F}_i$ that are constant with respect to the first argument, such that $f(x) = h(\tilde{x})$ and $g(x) = k(\tilde{x})$. Then by \eqref{proofsuffassoc}, we have 
\begin{equation}
\label{proofsuffassoc1}
T_{s,s+t} f g (x) \geq T_{s,s+t} f (x)\cdot T_{s,s+t} g(x).
\end{equation}
Note that we chose a fixed arbitrary $s\geq0$. We could follow the above procedure using any $s\geq0$, and thus we would obtain \eqref{proofsuffassoc1} for all $s,t\geq0$, giving us our desired result.
\end{proof}

We can now apply this to characterize association for jump-FEPs based on the time-dependent \Levy measures. 

\begin{theorem}
\label{thm:inhomoassocjump}
Let $X=(X_t)_{t\geq0}$ be a FEP with  FES $(T_{s,t})_{s\leq t}$, generators $(\mcl{A}_s)_{s\geq0}$ with rich domains, and that $C_c^\infty(\R^d)$ is the core for $\mcl{A}_s$, for all $s\geq0$. Let the corresponding symbols $p_s(x,\xi)$ be $s$-continuous and bounded 
with characteristic triplet $(b_s(x), 0, \nu_s(x,dy))$. If $X$ is stochastically monotone, then $X$ is spatially associated if and only if 
\begin{equation}
\label{resnicktimespace}
\nu_s(x, (\R_+^d \cup\R_-^d)^c) = 0, \hspace{.5cm} \forall s\geq0, \hspace{.1cm} x\in\R^d.
\end{equation}
\end{theorem}

\begin{proof}
($\Leftarrow$)  Assume \eqref{resnicktimespace}. Let $I(p_s)$ be the extended generator onto $C_b^2(\R^d)$, which is an integro-differential operator. Choose $s\geq0$, $f,g\in C_b^2(\R^d)\cap \mcl{F}_i$. Then
\begin{align*}
&I(p_s) fg(x) - f(x) I(p_s)g(x) - g(x) I(p_s) f(x) 
\\ & = \int_{\R^d\setminus\{0\}} (f(x+y) - f(x)) (g(x+y) - g(x)) \nu_s(x,dy)
\\ & =  \int_{\R_+^d\setminus\{0\}} (f(x+y) - f(x)) (g(x+y) - g(x)) \nu_s(x,dy)
\\ & \hspace{.3cm} +  \int_{\R_-^d\setminus\{0\}} (f(x+y) - f(x)) (g(x+y) - g(x)) \nu_s(x,dy)
\\ & \geq0.
\end{align*}
Then by Theorem \ref{thm:inhomoliggett}, $X$ is spatially associated.

\mypar

\noindent ($\Rightarrow$) We just show the proof for dimension $d=2$. Let $X$ be spatially associated. Then by Theorem \ref{thm:inhomoliggett}, $I(p_s) fg \geq f I(p_s)g + g I(p_s) f$ for all $s\geq0$, $f,g\in C_b^2(\R^2) \cap \mcl{F}_i$. This implies 
$$\int_{\R^2\setminus\{0\}} (f(x+y) - f(x)) (g(x+y) - g(x)) \nu_s(x,dy)\geq0, \hspace{.3cm} \forall s\geq0.$$

\noindent Assume for contradiction that there exists $t_0\geq0$ and $x=(x_1,x_2)\in\R^2$ such that $\nu_{t_0}(x,(\R_+^2\cup \R_-^2)^c)>0$. WLOG, say $\nu_{t_0}(x,(0,\infty) \times (-\infty,0))>0$. Then by continuity of measure, there exists $a>0$ such that  $\nu_{t_0}(x,(a,\infty) \times (-\infty,-a))>0$.  Fix $\epsilon>0$ and choose $f,g\in C_b^\infty(\R^2) \cap \mcl{F}_i$ such that 
\begin{equation*}
f(y_1,y_2)= \begin{cases} 
      0 & \textrm{ if \hspace{.1cm} $y_1\leq x_1+\epsilon a$} \\
      1 & \textrm{ if \hspace{.1cm} $y_1\geq x_1+ a$,} \\
   \end{cases} \quad \quad \quad
g(y_1,y_2)= \begin{cases} 
      0 & \textrm{ if \hspace{.1cm} $y_2\geq x_2-\epsilon a$} \\
      -1 & \textrm{ if \hspace{.1cm} $y_2\leq x_2-a$.} \\
   \end{cases} 
\end{equation*}

\noindent This implies $f(x)=g(x)=0$. Hence,
\begin{align*}
0 & \leq \int_{y\neq0} (f(x+y)- f(x))(g(x+y)-g(x)) \nu_{t_0}(x,dy)
\\ & = \int_{y\neq0} f(x+y)g(x+y) \nu_{t_0}(x,dy)
\\& = \int_{(a,\infty)\times (-\infty,-a)} f(x+y)g(x+y) \nu_{t_0}(x,dy)  + \int_{ (a,\infty)\times [-a,-\epsilon a]} f(x+y)g(x+y) \nu_{t_0}(x,dy)
\\ & \hspace{.2cm} + \int_{ [\epsilon a,a] \times (-\infty,-a)} f(x+y)g(x+y) \nu_{t_0}(x,dy) + \int_{ [\epsilon a,a] \times [-a,-\epsilon a]} f(x+y)g(x+y) \nu_{t_0}(x,dy)
\\ & = -\nu_{t_0}(x, (a,\infty)\times (-\infty,-a)) - \int_{(a,\infty)\times [-a,-\epsilon a]} g(x+y) \nu_{t_0}(x,dy)
\\ & \hspace{.2cm} +\int_{ [\epsilon a,a] \times (-\infty,-a)} f(x+y)\nu_{t_0}(x,dy)  + \int_{[\epsilon a,a] \times [-a,-\epsilon a]} f(x+y)g(x+y) \nu_{t_0}(x,dy)
\\ &\leq -\nu_{t_0}(x, (a,\infty)\times (-\infty,-a)),
\end{align*}   
which implies $\nu_{t_0}(x,(a,\infty)\times (-\infty,-a)) =0$, a contradiction. 
\end{proof}

\subsection{Other forms of dependence in FEPs}

In \cite{Tu2018a}, we showed that the \Levy measure condition  \eqref{resnickx}  was not only equivalent to spatial association for stochastically monotone jump-Feller processes, but also to spatial PUOD, PLOD, POD, PSD, PSA, and WA. These other forms of dependence can analogously be characterized in the time-inhomogeneous setting for the jump processes considered in Theorem \ref{thm:inhomoassocjump}, as was mentioned in Remark \ref{rem:spatassoc}(ii). To do this, we show that \eqref{resnicktimespace} is a necessary condition for spatial PUOD. Firstly, 

\begin{definition}
{\rm
Let $X=(X_t)_{t\geq0}$ be a time-inhomogeneous Markov process on $\R^d$. We say $X$ is \textbf{spatially  PUOD} if for every $s\leq t$, $x\in\R^d$, 
\begin{center}
$\E \left( \prod_{i=1}^d f_i (X_t^{(i)}) \left|X_s=x \right.\right) \geq \prod_{i=1}^d \E (f_i(X_t^{(i)})|X_s=x)$,
\end{center}
where $f_i:\R\rightarrow\R_+$ are non-decreasing. }
\end{definition}


\begin{theorem}
\label{thm:inhomoPUODnec}
Let $X=(X_t)_{t\geq0}$ be a FEP with FES $(T_{s,t})_{s\leq t}$, generators $(\mcl{A}_s)_{s\geq0}$ with rich domains, and that $C_c^\infty(\R^d)$ is the core for $\mcl{A}_s$, for all $s\geq0$. Let the corresponding symbols $p_s(x,\xi)$ be $s$-continuous and bounded 
with characteristic triplet $(b_s(x), 0, \nu_s(x,dy))$. If $X$ is spatially PUOD, then $\nu_s(x, (\R_+^d \cup\R_-^d)^c) = 0$,  $\forall s\geq0$, $x\in\R^d$. 
\end{theorem}

\begin{proof}
Let $\tilde{X} = (\tilde{X}_t)_{t\geq0}$ on $\R_+\times \R^d$ be the transformation of $X$, given by Prescription \ref{pre1}, which has Feller semigroup $(\tilde{T}_t)_{t\geq0}$, generator $(\mcl{\tilde{A}},\mcl{D}(\mcl{\tilde{A}}))$ with rich domain, bounded symbol $\tilde{p}(\tilde{x},\tilde{\xi})$, characteristics $(\tilde{b}(\tilde{x}),0 , \tilde{\nu}(\tilde{x},d\tilde{y}))$ and extended generator $I(\tilde{p})$ on $C_b^2(\R_+\times\R^d)$, as given to us by equations \eqref{eq:semigroup} to \eqref{eq:bddsymb}.

\par
Choose $\tilde{x}=(s,x)$. Let $f:\R_+\times \R^d\rightarrow\R_+$ defined by $f(x_0,...,x_d) = \prod_{i=0}^d f_i(x_i)$, where $f_i:\R\rightarrow\R_+$ are non-decreasing, for all $i$. Then 
\begin{align*}
\tilde{\E}^{\tilde{x}} f(\tilde{X}_t^{(0)},...,\tilde{X}_t^{(d)})  = \tilde{\E}^{\tilde{x}} f(\tilde{X}_t)
& = \E (f_{s+t}(X_{s+t})|X_s=x)
\\ & =  \E \left(f(s+t,X_{s+t}^{(1)},...,X_{s+t}^{(d)}) |X_s=x\right)
\\ & =  \E \left(f_0(s+t) f_1(X_{s+t}^{(1)})... f_d(X_{s+t}^{(d)}) |X_s=x\right)
\\ & \geq f_0(s+t) \cdot \prod_{i=1}^d \E ( f_i(X_{s+t}^{(i)}) | X_s=x)
\\ & = \prod_{i=0}^d \tilde{\E}^{\tilde{x}} f_i (\tilde{X}_t^{(i)})
\end{align*}
where we obtain the inequality by spatial PUOD of process $X$. Thus, the above calculation shows $\tilde{X}_t$ is PUOD for all $t\geq0$ in $\R_+\times\R^d$ with respect $\tilde{\dP}^{\tilde{x}}$, for all $\tilde{x}$. By Lemma \ref{PODnec}, we have that $\tilde{\nu}(\tilde{x}, (\R_+^{d+1} \cup \R_-^{d+1})^c) = 0$ for all $\tilde{x}\in\R_+\times\R^d$. Observe that the set $$\{0\}\times (\R_+^d \cup \R_-^d)^c \subseteq (\R_+^{d+1} \cup \R_-^{d+1})^c.$$
Hence, if $\tilde{x}=(s,x)$,
\begin{align*}
0  = \tilde{\nu}(\tilde{x}, (\R_+^{d+1} \cup \R_-^{d+1})^c)
 \geq \tilde{\nu}(\tilde{x}, \{0\}\times (\R_+^d \cup \R_-^d)^c )
& = \nu_s(x, (\R_+^d \cup \R_-^d)^c)\cdot \delta_0(\{0\})
\\& =  \nu_s(x, (\R_+^d \cup \R_-^d)^c)
\end{align*}
which implies $ \nu_s(x, (\R_+^d \cup \R_-^d)^c)=0$, completing our result.
\end{proof}

\begin{remark}
{\rm 
Theorem \ref{thm:inhomoPUODnec} also holds true if we replace ``PUOD" by ``PLOD". This can be easily verified by choosing $f_i:\R\rightarrow\R_+$ in the proof of Theorem \ref{thm:inhomoPUODnec} to be  non-increasing.}
\end{remark}

\begin{corollary}
\label{CorEquiv}
Let $X=(X_t)_{t\geq0}$ be a stochastically monotone FEP with the same assumptions as Theorem \ref{thm:inhomoPUODnec}. Then condition \eqref{resnicktimespace} is equivalent to $X$ being spatially associated, WA, PSA, PSD, POD, PUOD, PLOD. 
\end{corollary}

The equivalences in Corollary \ref{CorEquiv} are presented in Figure \ref{fig:dep_map}. The dashed lines are the implications proven in this paper in Theorems \ref{thm:inhomoassocjump} and \ref{thm:inhomoPUODnec}.
\begin{center}
\begin{figure}[h]
  \centering
  \includegraphics[trim = {0cm 0cm 0cm 0cm}, scale=.7]{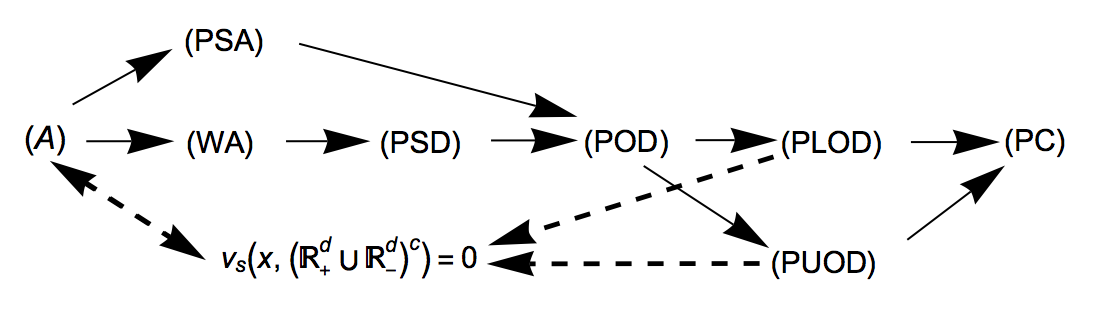}\\
  \caption{Equivalence of dependencies under condition \eqref{resnicktimespace} for FEPs}\label{fig:dep_map}
\end{figure}
\end{center}

\section{Applications and examples}
\label{sec:examples}
 
We first present in Section \ref{sec:additive} an important example of time-inhomogeneous Markov processes, called additive processes. These are also called time-inhomogeneous \Levy processes, and a nice sub-class of processes with independent increments (PII). Such processes are useful in financial models, such as stochastic volatility models with jumps (see \cite[Ch.15]{Cont2004}). In Section \ref{sec:comparison}, we show an application of the technique of transformation of time-inhomogeneous to time-homogeneous Markov processes in comparison theorems. 

\subsection{Additive processes}
\label{sec:additive}

A \textbf{process with independent increments (PII)} is a stochastic process $X=(X_t)_{t\geq0}$ on sample space $(\Omega,\mcl{G},(\mcl{G}_t)_{t\geq0},\dP)$ such that $X$ is c\`{a}dl\`{a}g, adapted, with $X_0=0$ a.s. and for all $s\leq t$, $X_t - X_s$ is independent of $\mcl{F}_s$. These processes and their semimartingale nature are be described in Jacod and Shiryaev \cite[Ch.II]{Jacod2003}. 

\begin{definition}
{\rm
If process $X=(X_t)_{t\geq0}$ on $\R^d$ is an \textbf{additive process} if it is a PII and satisfies stochastic continuity, i.e. $\displaystyle \lim_{h\searrow0} \dP(|X_{t+h} - X_t|\geq a) = 0$, for all $a>0$, $t\geq0$.
}
\end{definition}

Thus, observe that one can obtain additive processes by relaxing ``stationary increments" in the definition of a \Levy process. The following is a theorem found in Sato's book \cite{Sato1999} that tells us that additive processes still have ``infinitely divisible-like" behavior.

\begin{theorem}[Sato, \cite{Sato1999}, p.47]
Let $X=(X_t)_{t\geq0}$ be an additive process on $\R^d$. Then $X_t$ is infinitely divisible for all $t\geq0$, and $\phi_{X_t}(u) = \exp(p_t(u))$, where 
\begin{equation*}
p_t(u) = i u\cdot b_t  - \frac{1}{2} u\cdot \Sigma_t u + \int_{\R^d\setminus\{0\}} (e^{iu\cdot y} - 1 - iu\cdot y \chi(y)) \nu_t(dy)
\end{equation*}
is the symbol, where for all $t\geq0$, $\Sigma_t$ is a symmetric positive definite $d\times d$ matrix, $\nu_t$ is a \Levy measure, and $b_t\in \R^d$. 
\end{theorem}

Stochastic continuity of $X$ yields continuity in $t$ of characteristics $(b_t, \Sigma_t, \nu_t)$ and of the characteristic exponent $p_t$. 

\begin{theorem}[Sato, \cite{Sato1999}, p.52]
\label{propadditive}
An additive process $X$ with characteristics $(b_t, \Sigma_t, \nu_t)$ satisfies
\begin{itemize}[itemsep=.5pt]
\item Positiveness: $b_0=0$, $\Sigma_0=0$, $\nu_0=0$, and for all $s\leq t$, $\Sigma_t - \Sigma_s$ is a positive definite matrix, and $\nu_t(B) \geq \nu_s(B)$ for all $B\in\mcl{B}(\R^d)$.
\item Continuity: if $s\rightarrow t$, then $\Sigma_s\rightarrow \Sigma_t$, $b_s\rightarrow b_t$, and $\nu_s(B) \rightarrow \nu_t(B)$ for all $B\in\mcl{B}(\R^d)$ such that $B \subseteq \{x: |x|\geq \epsilon\}$ for some $\epsilon>0$.
\end{itemize}
\end{theorem}

\begin{corollary}
Let $X$ be an additive process with characteristic exponents $p_t$. Then $p_t(u)$ is continuous in $t$ for all $u\in\R^d$.
\end{corollary}

Additive processes can also be viewed from the perspective of Markov processes. These processes are time-inhomogeneous, spatially-homogeneous Markov processes, with Markov evolution $(T_{s,t})_{s\leq t}$ given by 
\begin{equation}
\label{evoadditive}
T_{s,t} f(x) = \E (f(X_t)|X_s=x) = \E f(X_t - X_s +x).
\end{equation}
Such Markov evolutions are also strongly continuous on $C_0(\R^d)$. 

\begin{theorem}
Let $X$ be an additive process with Markov evolution $(T_{s,t})_{s\leq t}$ defined by \eqref{evoadditive}. Then $(T_{s,t})_{s\leq t}$ is strongly continuous, thus making $(T_{s,t})_{s\leq t}$  a Feller evolution system.
\end{theorem}

\begin{proof}
For a proof, see \cite[Thm. 4.13]{TuThesis}. 
\end{proof}

Thus, additive processes are a subclass of Feller evolution processes. It is shown in Cont and Tankov \cite[Ch.14]{Cont2004} that the generators $\mcl{A}_s$ of an additive process has the form 
\begin{equation}
\label{generatoradditive}
\mcl{A}_s f (x) = b_s\cccdot \nabla f(x) + \frac{1}{2} \nabla \cccdot \Sigma_s \nabla f(x) + \int_{\R^d\setminus\{0\}} (f(x+y) - f(x) -y\cccdot f(x) \chi(y))\nu_s(dy)
\end{equation}
for $f\in C_0^2(\R^d)$. Thus the symbol of the operator $\mcl{A}_s$ coincides with the characteristic exponent $p_s(\xi)$, which is analogous to the relationship between symbols and characteristic exponents of \Levy processes. Hence, the additive process has an extended generator, which is an integro-differential operator $I(p_s)$ on $C_b^2(\R^d)$ defined by the RHS of \eqref{generatoradditive}. Therefore, additive processes are FEPs with symbols $p_s(\xi)$ and characteristics $(b_s, \Sigma_s, \nu_s)$ that do not depend on $x$, i.e. the state space. They can be classified as FEPs that are spatially homogeneous.

\par

Moreover, their FESs $(T_{s,t})_{s\leq t}$ are always stochastically monotone: if $x\leq y$ and $f\in B_b(\R^d)\cap \mcl{F}_i$, then
$$T_{s,t} f(x) = \E f(X_t - X_s +x) \leq \E f(X_t-X_s + y) = T_{s,t} f(y).$$
Hence, we can apply Theorems \ref{thm:inhomoliggett}-\ref{thm:inhomoPUODnec} to additive processes! 

\begin{theorem}
\label{thm:resnickadditive}
Let $X=(X_t)_{t\geq0}$ be an additive process with symbols $p_s(\xi)$ and characteristic triplets $(b_s, 0, \nu_s)$. Then $X$ is spatially associated if and only if 
\begin{equation}
\label{eq:resnickadditive}
\nu_s((\R_+^d \cup \R_-^d)^c)=0, \hspace{.4cm} \forall s\in\Q_+.
\end{equation}
\end{theorem}

\begin{proof}
Notice that this is a slightly weaker assumption than the statement of Theorem \ref{thm:inhomoassocjump}. This is because in the case of additive processes, $\nu_s((\R_+^d \cup \R_-^d)^c)=0$,  $\forall s\in\Q_+$ implies $\nu_s((\R_+^d \cup \R_-^d)^c)=0$,  $\forall s\in\R_+$. We show this in $d=2$. 

\par

 Assume for contradiction that there is $t_0\in \R_+\setminus\Q_+$, such that $\nu_{t_0} ((\R_+^2 \cup\R_-^2)^c)>0$. WLOG, say $\nu_{t_0} ((0,\infty)\times (-\infty,0))>0$. By continuity of measure, there exists $a>0$ such that $\nu_{t_0}((a,\infty)\times (-\infty,-a))>0$. By Theorem \ref{propadditive}, since $A=(a,\infty)\times (-\infty,-a)$ is bounded away from $0$, there exists $(t_n)_{n\in\N}\subset \Q_+$ such that $t_n\rightarrow t_0$ and $$\nu_{t_n}((a,\infty)\times (-\infty,-a))\rightarrow
\nu_{t_0} ((a,\infty)\times (-\infty,-a))>0, \hspace{.5cm} \text{as} \hspace{.5cm}  n\rightarrow\infty.$$
Therefore, there exists $N$ large such that for all $n\geq N$, $\nu_{t_n} ((a,\infty)\times (-\infty,-a))>0$, which is a contradiction.  Hence, $\nu_t((\R_+^2 \cup \R_-^2)^c)=0$ for all $t\geq0$, which is equivalent to $X$ being spatially associated by Theorem \ref{thm:inhomoassocjump}. 
\end{proof}

\begin{corollary}
Let $X=(X_t)_{t\geq0}$ be an additive process with symbols $p_s(\xi)$ and characteristic triplets $(b_s, 0, \nu_s)$. Then $X$ is spatially PUOD (and also PLOD, POD, PSD, PSA, WA) if and only if $\nu_s$ satisfies \eqref{eq:resnickadditive}. 
\end{corollary}

\begin{proof}
The corollary is a direct result of Theorems \ref{thm:inhomoPUODnec} and  \ref{thm:resnickadditive} 
\end{proof}


An example of an interesting additive process to which these dependence results apply is the following L\'{e}vy-driven volatility model:

\begin{example}
{\rm Let $L=(L_t)_{t\geq0}$ be a \Levy process in $\R^d$ with triplet $(b, 0, \nu)$. Let $\sigma:\R_+\rightarrow\R_+^{d\times d}$. The process $X=(X_t)_{t\geq0}$ defined by 
\begin{center}
$X_t = \int_0^t \sigma(s) dL_s$
\end{center}
is called a \textbf{\Levy process with deterministic volatility} and is an additive process.

} 
\end{example}

\subsection{Comparison of Markov processes}
\label{sec:comparison}

Some of the techniques we employed in Section \ref{sec:mainresults} to prove results on dependence structures can also be used to prove comparison theorems of certain Markov processes. Let $\mcl{F}$ be a cone of functions, such as $\mcl{F}_i=\{f:\R^d\rightarrow\R, f \text{ non-decreasing}\}$ (for more on cones $\mcl{F}$, see \cite{Ruschendorf2016}). For time-homogeneous Markov processes $X$ and $Y$ with semigroups $(S_t)_{t\geq0}$ and $(T_t)_{t\geq0}$ and generators $\mcl{A}$ and $\mcl{B}$, respectively, we say that \textbf{$Y$ dominates $X$ with respect to $\mcl{F}$} if $S_t f \leq T_t f$, for all $t\geq0$ and $f\in \mcl{F}$.  

For time-inhomogeneous Markov processes $X$ and $Y$, with Markov evolutions $(S_{s,t})_{s\leq t}$ and $(T_{s,t})_{s\leq t}$, we say \textbf{$Y$ dominates $X$ with respect to $\mcl{F}$} if $S_{s,t}f \leq T_{s,t} f$ for all $s\leq t$ and all $f\in \mcl{F}$. R\"{u}schendorf  has proven comparison theorems for general Markov processes which are time-homogeneous (2008) \cite{Ruschendorf2008} and time-inhomogeneous (2016) \cite{Ruschendorf2016}. These sufficient conditions were based on the generators of the Markov process. We show that in the case of two rich Feller processes, sufficient conditions for domination can be given using the extended generator $I(p)$ (Theorem \ref{thm:tufellercomp}).  Then we use that result and the technique Prescription \ref{pre1} to obtain a nice comparison theorem for time-inhomogeneous Feller evolutions systems (Theorem \ref{thm:mycomparison}). We  consider the cone of $\mcl{F}_i$ in these theorems.

\begin{theorem}
\label{thm:tufellercomp}
If $X$ and $Y$ are rich Feller processes and have symbols $p^X$ and $p^Y$, respectively, then if $S_t f\in\mcl{F}_i$ for $f\in C_b(\R^d) \cap \mcl{F}_i$ and $I(p^X) f \leq I(p^Y) f$ for all $f\in C_b^2(\R^d)\cap\mcl{F}_i$, then $S_t f\leq T_t f$ for all $f\in C_b(\R^d)\cap \mcl{F}_i$.
\end{theorem}

\begin{proof}
Pick $f\in C_b^2(\R^d)\cap \mcl{F}_i$. Define $F:[0,\infty)\rightarrow C_b(\R^d)$ and $G:[0,\infty)\rightarrow C_b(\R^d)$ by
\begin{equation*}
F(t):= T_t f - S_t f \hspace{.5cm} \text{ and } \hspace{.5cm} G(t):= F'(t) - I(p^Y)F(t) = (I(p^Y) - I(p^X))S_t f.
\end{equation*}
$G(t)\geq0$, since $S_t f\in C_b^2(\R^d)\cap \mcl{F}_i$ and by our assumption. Thus since $F,G$ are continuous (wrt locally uniform convergence), then by Theorem \ref{extcauchy}, 
$$F(t) = T_t F(0) + \int_0^t T_{t-r} G(r) dr = \int_0^t T_{t-r} G(r) dr\geq0,$$
giving us our desired result.
\end{proof}


\begin{theorem}
\label{thm:mycomparison}
Let $X$ and $Y$ be Feller evolution processes with FESs $(S_{s,t})_{s\leq t}$ and $(T_{s,t})_{s\leq t}$, generators $(\mcl{A}_s)_{s\geq0}$ and $(\mcl{B}_s)_{s\geq0}$ with rich domains, symbols $p_s(x,\xi)$ and $q_s(x,\xi)$ that are $s$-continuous and bounded as in \eqref{inhomocont} and \eqref{inhomobounded}, respectively. Let $C_c^\infty(\R^d)$ be a core for the domains of the generators. Then if $X$ is stochastically monotone (wrt $\mcl{F}_i$), and 
\begin{center}
$I(p_s) f \leq I(q_s)f$, \hspace{.4cm} for all $f\in C_b^2(\R^d)\cap\mcl{F}_i$, for all $s\geq0$,
\end{center} 
then $S_{s,t} f \leq  T_{s,t} f$ for all $f\in C_b(\R^d) \cap \mcl{F}_i$, for all $s\leq t$. 
\end{theorem}

\begin{proof}
Let $\tilde{X} = (\tilde{X}_t)_{t\geq0}$ and $\tilde{Y} = (\tilde{Y}_t)_{t\geq0}$ on $\R_+\times \R^d$ be transformations of $X$ and $Y$, given by Prescription \ref{pre1}, which have Feller semigroups $(\tilde{S}_t)_{t\geq0}$ and  $(\tilde{T}_t)_{t\geq0}$, generators $(\mcl{\tilde{A}},\mcl{D}(\mcl{\tilde{A}}))$ and $(\mcl{\tilde{B}},\mcl{D}(\mcl{\tilde{B}}))$ with rich domains, bounded symbols $\tilde{p}(\tilde{x},\tilde{\xi})$ and $\tilde{q}(\tilde{x},\tilde{\xi})$, and extended generators $I(\tilde{p})$ and $I(\tilde{q})$ 
respectively, as given to us by equations \eqref{eq:semigroup} to \eqref{eq:bddsymb}.

Observe that for all $f\in C_b^2(\R_+\times\R^d)\cap\mcl{F}_i$, we have
\begin{equation}
\label{eq1:thmcomp}
\begin{split}
I(\tilde{p}) f(\tilde{x}) & = \frac{\partial}{\partial s} f(s,x) + I(p_s) f_s(x)
  \leq  \frac{\partial}{\partial s} f(s,x) + I(q_s) f_s(x)
 = I(\tilde{q}) f(\tilde{x}).
\end{split}
\end{equation}
Now let $h\in C_b^2(\R_+\times\R^d)\cap\mcl{F}_i$. Fix $s\geq0$. Then $\tilde{S}_t h|_{\{s\}\times\R^d} \in C_b^2(\{s\}\times\R^d) \cap \mcl{F}_i$ since $X$ is stochastically monotone. Then there exists $v\in C_b^2(\R_+\times\R^d)\cap \mcl{F}_i$ that is constant wrt first argument in $\R_+$ and $v(s,x) =\tilde{S}_t h|_{\{s\}\times\R^d}(s,x)$, for all $x\in\R^d$. Then by \eqref{eq1:thmcomp}, $I(\tilde{p})v \leq I(\tilde{q})v$. This implies that on $\tilde{x} = (s,x)$, 
\begin{equation}
\label{eq2:thmcomp}
I(\tilde{p})(\tilde{S}_t h|_{\{s\}\times\R^d})(s,x) \leq I(\tilde{q})(\tilde{S}_t h|_{\{s\}\times\R^d})(s,x). 
\end{equation}
Define $F,G:[0,\infty)\rightarrow C_b(\R_+\times\R^d)$ be defined by 
\begin{equation*}
F(t):= \tilde{T}_t h - \tilde{S}_t h \hspace{.5cm} \text{ and } \hspace{.5cm} G(t):= F'(t) - I(\tilde{q}) F(t) = (I(\tilde{q}) - I(\tilde{p}))\tilde{S}_t h
\end{equation*}

which are both continuous with respect to locally uniform convergence. Then by Theorem \ref{extcauchy}, $F(t) = \int_0^t \tilde{T}_{t-r} G(r) dr$. Hence, on $\tilde{x} = (s,x)$, $G(r) (\tilde{x}) = (I(\tilde{q}) - I(\tilde{p}))\tilde{S}_t h|_{\{s\}\times\R^d} (s,x) \geq0$ by \eqref{eq2:thmcomp}. Thus, $F(t)(\tilde{x})\geq0$. This implies $S_{s,s+t} h_{s+t}(x)\leq T_{s,s+t} h_{s+t}(x)$ for all $x\in\R^d$. Let $f\in C_b^2(\R^d)\cap\mcl{F}_i$. Then choose $h\in C_b^2(\R_+\times\R^d)\cap\mcl{F}_i$ that is constant in the first argument, and $h(\tilde{x})=f(x)$. Then we have $S_{s,s+t} f(x) = S_{s,s+t} h_{s+t}(x)\leq T_{s,s+t} h_{s+t}(x) = T_{s,s+t} f(x),$
giving us our desired result.
\end{proof}

 For more on comparison theorems of Markov processes, see R\"{u}schendorf \cite{Ruschendorf2008, Ruschendorf2016}.

\mypar 

\textbf{Acknowledgements}: The author would like to thank Dr. Jan Rosinski for his helpful advice and guidance regarding the ideas of this paper.

\newpage

\bibliographystyle{acm}

\bibliography{references-paper2_tu}

\end{document}